\documentclass[12pt]{amsart}

\usepackage{amsmath, amssymb, amscd, verbatim, xspace,amsthm}
\usepackage{latexsym, epsfig, color}
\usepackage[OT2,T1]{fontenc}
\DeclareSymbolFont{cyrletters}{OT2}{wncyr}{m}{n}
\DeclareMathSymbol{\Sha}{\mathalpha}{cyrletters}{"58}

\newcommand{\Z}{\ensuremath{{\mathbb{Z}}}\xspace}
\renewcommand{\P}{\ensuremath{{\mathbb{P}}}}

\newcommand{\Q}{\ensuremath{{\mathbb{Q}}}}

\newcommand{\E}{\ensuremath{{\mathbb{E}}}}

\newcommand{\ra}{\rightarrow}

\newcommand\Hom{\operatorname{Hom}}

\newcommand\Aut{\operatorname{Aut}}

\newcommand\im{\operatorname{im}}

\newcommand\Sur{\operatorname{Sur}}

\newcommand\tensor{\otimes}
\newcommand\isom{\simeq}
\newcommand\sub{\subset}

\newcommand\cok{\operatorname{cok}}

\newcommand\cO{\mathcal{O}}

\newcommand\bq{\begin{equation}}
\newcommand\eq{\end{equation}}

\newtheorem{proposition}{Proposition}[section]
\newtheorem{theorem}[proposition]{Theorem}
\newtheorem{corollary}[proposition]{Corollary}

\newtheorem{lemma}[proposition]{Lemma}

\newtheorem{conjecture}[proposition]{Conjecture}

\theoremstyle{remark}
\newtheorem{remark}[proposition]{Remark}
\usepackage{fullpage,pdfsync}

\newenvironment{definition}{\vspace{2 ex}{\noindent{\bf Definition. }}}{\vspace{2 ex}}

\newtheorem{nts}{Note to self}

\title{Random integral matrices and the Cohen Lenstra Heuristics}

\author{Melanie Matchett Wood}
\address{Department of Mathematics\\
University of Wisconsin-Madison \\ 480 Lincoln Drive \\
Madison, WI 53705 USA\\
and
American Institute of Mathematics\\600 East Brokaw Road\\
San Jose, CA 95112 USA} 
\email{mmwood@math.wisc.edu}

\begin{document}
\begin{abstract}
We prove that given any $\epsilon>0$, random integral $n\times n$ matrices with independent entries that lie in any residue class modulo a prime with probability at most $1-\epsilon$  have cokernels asymptotically (as $n\ra\infty$) distributed as in the distribution on finite abelian groups that Cohen and Lenstra conjecture as the distribution for class groups of imaginary quadratic fields. 
This is a refinement of a result on the distribution of ranks of random matrices with independent entries in $\Z/p\Z$.
This is interesting especially in light of the fact that these class groups are naturally cokernels of square matrices.
 We also prove the analogue for $n\times (n+u)$ matrices.
\end{abstract}

\maketitle

\section{Introduction}
The Cohen-Lenstra heuristics are conjectures made by Cohen and Lenstra \cite{Cohen1984} on the distribution of class groups of quadratic number fields.  For a prime $p,$ we write $G_p$ for the Sylow $p$-subgroup of an abelian group $G$.  We write $\operatorname{Cl}(K)$ for the class group of a number field $K$.

\begin{conjecture}[Cohen and Lenstra \cite{Cohen1984}
]\label{C:CL}
Let $p$ be an odd prime.
Let $S^-_X$ be the set of negative fundamental discriminants $D\geq -X$.  Let $p$ be an odd prime and $B$ be a finite abelian $p$-group.
Then
$$
\lim_{X \ra\infty} \frac{\#\{D\in S^-_X \,|\, \operatorname{Cl}(\Q(\sqrt{D})_p \isom B \} }{|S^-_X|} =\frac{\prod_{k=1}^\infty (1-p^{-k})}{|\Aut(B)|}.
$$
\end{conjecture}

Friedman and Washington \cite{Friedman1989} show that if $H_n \in M_{n\times n}(\Z_p)$ is a random matrix drawn with respect to Haar measure on the space of $n\times n$ matrices over $\Z_p$, then
\begin{equation}\label{E:FW}
\lim_{n\ra\infty} \P(\cok(H_n) \isom B) =\frac{\prod_{k=1}^\infty (1-p^{-k})}{|\Aut(B)|}.
\end{equation}
In other words, cokernels of random $p$-adic square matrices drawn with respect to Haar measure are distributed 
according to Cohen and Lenstra's conjectured distribution of class groups (asymptotically as the size of the matrices grows).

Let $K=\Q(\sqrt{D})$ for some $D\in S^-_X$, and $S$ be any finite set of primes of $K$ that generate $\operatorname{Cl}(K)$.  
We write $\cO_S^*$ for the $S$-units in the integers $\cO_K$, and $I_K^{S}$ for the abelian group of fractional ideals generated by the elements of $S$.
Then 
\begin{equation}\label{E:casm}
\operatorname{Cl}(K)=\cok( \cO_S^* \ra I_K^{S} ),
\end{equation}
where the map takes $\alpha\mapsto (\alpha)$.
So $\operatorname{Cl}(K)_p=\cok( \cO_S^* \tensor_\Z \Z_p \ra I_K^{S} \tensor_\Z \Z_p)$.
Since  $I_K^{S}$ and $\cO_S^*$ are both free abelian groups of rank $|S|$ (when $-D> 4$), we have written 
$\operatorname{Cl}(K)_p$ as a cokernel of a $p$-adic square matrix $R_D\in M_{n\times n}(\Z_p)$.
As $D$ varies, we have a random  $p$-adic square matrix $R_D\in M_{n\times n}(\Z_p)$ (where $n$ depends on $D$ and we can take $n\ra \infty$) and Conjecture~\ref{C:CL} is a statement about the distribution of the cokernels of the random matrices $R_D$.

One might thus imagine that there could be some sense in which the $R_D$ become equidistributed with respect to Haar measure, and that this would imply Conjecture~\ref{C:CL}.  However, in this paper we show that in fact having cokernels distributed according to Cohen and Lenstra's conjectured distribution of class groups is a rather robust feature of random matrix regimes, and so much weaker statements (than Haar equidistribution) about the distribution $R_D$ would also imply Conjecture~\ref{C:CL}.  

As a particular example, if we take $X_n\in M_{n\times n}(\Z_p)$ whose entries are  independent and $0$ with probability $q$ and $1$ with probability $1-q$, then (for any $q$ and) for every $p$, we have Equation~\eqref{E:FW} with $H_n$ replaced by $X_n$.  These $X_n$, with their entries concentrated in $\{0,1\}$, are nowhere near Haar equidistributed in any $\Z_p$, yet they still have the same cokernel distributions as Haar equidistributed random matrices.  More generally, we have the following.

\begin{theorem}\label{T:MainIntro}
Let $p$ be a prime and $\epsilon>0$, and for each $n$ let $X_n \in M_{n\times n} (\Z_p)$ be a random matrix with independent entries.  Further, for any entry $(X_n)_{i,j}$ and any  $r\in \Z/p\Z$, we require that $\P((X_n)_{i,j} \equiv r \pmod{p})\leq 1-\epsilon$.  Then for any finite abelian $p$-group $B$,
\begin{equation*}
\lim_{n\ra\infty} \P(\cok(X_n) \isom B) =\frac{\prod_{k=1}^\infty (1-p^{-k})}{|\Aut(B)|}.
\end{equation*}
\end{theorem}

Note that the matrix entries are not required to be identically distributed and can vary with $n$.  Of course, some condition that the matrix entries are not too concentrated, like $\P((X_n)_{i,j} \equiv r \pmod{p})\leq 1-\epsilon$, is certainly necessary, since if the matrices had even two rows whose values were all $r \pmod{p},$ then $\cok(X_n)$ could never be the trivial group.

In fact, in Corollary~\ref{C:Main}, we prove a statement about random integral matrices that implies Theorem~\ref{T:MainIntro}, determining not only the Sylow $p$-subgroups of their cokernels a single $p$, but rather the Sylow $p$-subgroups of their cokernels simultaneously for any finite set of $p$, and we see (as Cohen and Lenstra \cite{Cohen1984} predict for class groups) that the Sylow $p$-subgroups for different $p$ are independent.

Of course, the independence of the matrix entries in Theorem~\ref{T:MainIntro} in a significant hypothesis (and  not true in such a form for class groups), and one might wonder to what extent it is necessary.  In \cite{Wood2014}, it is shown that if one takes the matrices $X_n$ symmetric, but with otherwise independent entries, their cokernels  
have a \emph{different}
 distribution that that in Theorem~\ref{T:MainIntro}.  The work in that paper was to determine the distribution of Jacobians (a.k.a. sandpile groups) of random graphs, which are a more accessible analogue of class groups.  That application also required dealing with the fact that each diagonal entry of the relevant matrix (the graph Laplacian) is dependent on all the entries in its column, and this ``small'' dependence of the diagonal did not have an effect on the cokernel distribution.  

In fact, Cohen and Lenstra \cite{Cohen1984} also make conjectures about class groups of real quadratic (and other totally real abelian) number fields.  In particular, if $S^+_X$ is the set of positive fundamental discriminants $D\leq X$, they conjecture
$$
\lim_{X \ra\infty} \frac{\#\{D\in S^+_X \,|\, \operatorname{Cl}(\Q(\sqrt{D})_p \isom B \} }{|S^-_X|} =\frac{\prod_{k=1}^\infty (1-p^{-k-1})}{|B||\Aut(B)|}.
$$
We see from equation~\eqref{E:casm} that these class groups are cokernels of $n\times (n+1)$ matrices, since $\cO_S^*$ will have rank $|S|+1$ when the number field $K$ is real quadratic.  We in fact prove the following, which follows from Corollary~\ref{C:Main}.

\begin{theorem}
\label{T:u}
For any $u\geq 0$,  for random $X_n\in M_{n\times (n+u)}(\Z_p)$ with entries as in  Theorem~\ref{T:MainIntro}, 
\begin{equation*}
\lim_{n\ra\infty} \P(\cok(X_n) \isom B) =\frac{\prod_{k=1}^\infty (1-p^{-k-u})}{|B|^u|\Aut(B)|}.
\end{equation*}
\end{theorem} 
These distributions on finite abelian groups for other $u$ also arise in the general theory Cohen and Lenstra build to formulate their conjectures.

While the results of this paper are particularly notable for their connection to the Cohen Lenstra heuristics, the proofs in this paper and the history of previous work lie in the fields of additive combinatorics and probability.  For $X_n\in M_{n\times n}(\Z_p)$, then $\cok(X_n)$ is trivial if and only if $X_n$ is a non-singular matrix when reduced mod $p$.  More generally, the corank of the matrix mod $p$ is the rank of the cokernel.  There is a long history of work on singularity and ranks of the random matrices we consider above mod $p$,
 including results of Kozlov \cite{Kozlov1966}, Kovalenko and Levitskaja  \cite{Kovalenko1975a}, Charlap, Rees, and Robbins \cite{Charlap1990} (first proving Theorem~\ref{T:MainIntro} in the case that $B$ is the trivial group),  Kahn and Koml\'{o}s  \cite{Kahn2001}, and Maples \cite{Maples2010}  for
 for general $p$.
However, even our result on ranks (Corollary~\ref{C:prank}), that for $X_n$ as in Theorem~\ref{T:MainIntro}, 
$$
\lim_{n\ra \infty} \P(\operatorname{rank}(X_n)=n-k)=p^{-k^2}\prod_{i=1}^{k} (1-p^{-i})^{-2} \prod_{i\geq 1} (1-p^{-i}) 
$$
appears to be new with our hypotheses.  The realization that the cokernel distribution, and not just the ranks, should also be insensitive to the distributions of the entries of the matrices is due to Tao and Maples (see \cite{Maples2013} for some interesting work towards Theorem~\ref{T:MainIntro}).

This robustness of certain statistics of random matrices under changes to the entry distribution of the matrices is an important theme in the study of random matrices and is called \emph{universality} of those statistics.  
For example, the best upper bounds on the singularity probability of discrete random matrices with independent entries in characteristic $0$, due to Bourgain, Vu, and Wood \cite{Bourgain2010}, are insensitive to the actual values the entries take (as long as they are not too concentrated).  

To prove our main result, we first determine the moments of the cokernel distributions, and from that determine the distributions themselves.  Our specific approach was developed in \cite{Wood2014} for the case of symmetric matrices.  In this paper, we are able to use a much simplified version of that in \cite{Wood2014} since our matrices have all their entries independent.  To find the moments $\E(\#\Sur(\operatorname{cok}(X_n),G))$,  we prove inverse Littlewood-Offord theorems (Lemmas~\ref{L:Fcodecolumn} and \ref{L:probdepthestimatecolumn}).  These both say that if values of several linear functions of our $n$ independent variables are not close to equidistributed, then the linear functions are close to having extra structure. The extra structure is analogous to having a linear dependence in rows of a matrix after deleting a small number of columns, but since our linear algebra is not always over a field there are many layers to the type of dependence we can have, which are captured by our notion of \emph{depth}.  
Inverse Littlewood-Offord Theorems 
are a key component in the most recent work on singularity probability of discrete random matrices in characteristic $0$, both \cite{Bourgain2010} mentioned above and the earlier work of Tao and Vu \cite{Tao2007}.  
(See the papers of Tao and Vu \cite{Tao2010b} and Nguyen and Vu for the most recent \cite{Nguyen2011} inverse Littlewood-Offord Theorems in characteristic $0$, as well as a guide to the extensive previous work on the problem.)  
However, there are significant differences in the actual mathematics of these theorems in characteristic $0$ versus  characteristic $p$, since in characteristic $0$ one doesn't expect any kind of equidistribution, but rather just a good upper bound on the probabilities.  Maples  \cite{Maples2013} proves a Littlewood-Offord Theorem in characteristic $p$ that is not strong enough for our purposes for fixed $p$, but does have the advantage of uniformity in $p$.

To finally determine our cokernel distribution from the moments, we can't rely on the usual probabilistic methods such as Carleman's condition (since our moments are too big--our $k$th moment is of order $p^{k^2/2}$).  However, we use a specifically tailored result \cite{Wood2014} that in our cases shows that that moments we obtain determine a unique distribution.  This  situation of needing to show fast growing moments of random abelian groups determine the distribution of the groups has arisen before in number theory, both in Cohen-Lenstra problems, e.g. in the work of  Fouvry and Kl\"{u}ners \cite{Fouvry2006} and Ellenberg, Venkatesh, and Westerland \cite{Ellenberg2009}, and in a related problem about Selmer groups in work of Heath-Brown \cite{Heath-Brown1994}.  Ellenberg, Venkatesh, and Westerland \cite{Ellenberg2009} make progress towards proving the function field analogue of the Cohen-Lenstra heuristics by proving new homological stability theorems that determine some of the  moments of the relevant class groups.

\subsection{Further notation}\label{S:not}
We use $[n]$ to denote $\{1,\dots,n\}$. We write $\Hom(A,B)$ and  $\Sur(A,B)$ for the set of homomorphisms and surjections, respectively, from $A$ to $B$.  We write $\P$ for probability and $\E$ for expected value. 
We write $\exp(x)$ for the exponential function $e^x$.

Throughout the paper, we let $a$ be a positive integer and let $R=\Z/a\Z$.  We then study finite abelian groups $G$ whose exponent divides $a$, i.e. $aG=0$.  We write $G^*$ for $\Hom(G,R)$.

\section{Finding the moments}\label{S:mom}
We will study integral matrices by reducing them mod $a$ for all $a$ (and analogously matrices over $\Z_p$ by reducing them mod $p^k$ for all $k$).
For the rest of the paper, let $a$ be a positive integer and let $R=\Z/a\Z$.  We change the notation for our random matrix  from the introduction to make it easier to read our proof.  
Let $M$ be a random $n\times (n+u)$ matrix with entries in $R$.  We let $M_1,\dots, M_{n+u}\in R^n$ be the columns of $M$, and $m_{ij}$ the entries of $M$ (so that the entries of $M_j$ are $m_{ij}$).  We let $\epsilon>0$.

The following definition captures the two hypotheses of Theorem~\ref{T:MainIntro}: independence of entries and entries not too concentrated.

\begin{definition}
A random variable $y$ in a ring $T$ is \emph{$\epsilon$-balanced} if for every maximal ideal $\wp$ of $T$ 
and $r \in T/\wp$ we have $\P(y \equiv r \pmod{\wp})\leq 1-\epsilon$ (e.g., 
$y\in R$ is \emph{$\epsilon$-balanced} if for every prime $p\mid a$ and $r \in \Z/p\Z$ we have $\P(y \equiv r \pmod{p})\leq 1-\epsilon$).
A random vector or matrix is \emph{$\epsilon$-balanced} if its entries are independent and $\epsilon$-balanced
\end{definition}

We let $V=R^n$ with basis $v_i$ and $W=R^m$ with basis $w_j$. 
Note for $\sigma\sub [n]$, $V$ has distinguished submodules $V_{\setminus \sigma}$ generated by the $v_i$ with $i\not\in\sigma$. (So $V_{\setminus \sigma}$ comes from not using the $\sigma$ coordinates.)
 We view $M\in \Hom(W,V)$ and it's columns $M_j$ as elements of $V$ so that
$M_j=Mw_j=\sum_i m_{ij}v_i. $  Let $G$ be a finite abelian group with exponent dividing $a$.
We have $\cok M=V/MW$.
To investigate the moments 
$
\E(\#\Sur(\cok M, G)),
$
we recognize that each such surjection lifts to a surjection $V\ra G$ and so we have
\begin{equation}\label{E:expandF}
\E(\#\Sur(\cok M, G))=\sum_{F\in \Sur(V,G)}  \P(F(MW)=0).
\end{equation}
If $M$ is $\epsilon$-balanced, then by the independence of columns, we have
$$
\P(F(MW)=0)=\prod_{j=1}^m \P(F(M_j)=0).
$$
So we aim to estimate these probabilities $\P(F(M_j)=0)$.
We will first estimate these for the vast majority of $F$, which satisfy the following helpful property.  

\begin{definition}
We say that $F\in \Hom(V,G)$ is a \emph{code} of distance $w$, if for every $\sigma\sub [n]$ with $|\sigma|<w$, we have $FV_{\setminus \sigma}=G$.
In other words, $F$ is not only surjective, but would still be surjective if we throw out (any) fewer than $w$ of the standard basis vectors from $V$.  (If $a$ is prime so that $R$ is a field, then this is equivalent to whether the transpose map $F: G^* \ra V^*$ is injective and has image $\im(F)\sub V^*$ a linear code of distance $w$, in the usual sense.)
\end{definition}

\begin{lemma}\label{L:Fcodecolumn}
Let $R$ and $G$ be as above.   Let $\epsilon >0$ and $\delta>0$.
Let $X$ be an $\epsilon$-balanced  random vector in $V$.
 Let $F\in \Hom(V,G)$ be a code of distance $\delta n$ and $A\in G$.
   For all $n$ we have
$$
\left |  \P(FX=A) - |G|^{-1} \right| \leq \exp(-\epsilon\delta n/a^2).
$$
\end{lemma}

Let $\zeta$ be a primitive $a$th root of unity.
To prove Lemma~\ref{L:Fcodecolumn}, we will use the discrete Fourier transform and the following basic estimate.

\begin{lemma}\label{L:singleest}
Let $y$ be an entry of an $\epsilon$-balanced random variable in $R$, and let $m$ be an integer such that $\zeta^m\neq 1$.  Then  $|\E(\zeta^{ my})|\leq \exp(-\epsilon /a^2)$.
\end{lemma}
\begin{proof}
This is proven in  \cite[Proof of Lemma 4.1]{Wood2014a}.  Briefly, the longest $|\E(\zeta^{ y})|$ could
be was if $\zeta^{y}$ was one $a$th root of unity $1-\epsilon$ of the time, and a consecutive (around the unit circle) $a$th root of unity the rest of the time.
\end{proof}

\begin{proof}[Proof of Lemma~\ref{L:Fcodecolumn}]
 We have, by the discrete Fourier transform,
$$
\P(FX=A)= |G|^{-1}\sum_{C\in G^*} \E(\zeta^{C(FX-A)})=|G|^{-1}+|G|^{-1} \sum_{C\in G^*\setminus \{0\}} \E (\zeta^{C(-A)})\prod_{1\leq i\leq n} \E(\zeta ^{C(v_i)X_{i}}  ).
$$
Since $C\ne 0$ and $F$ is a code, there must be at least $\delta n$ values of $i$ such that $F(v_i)\not\in \ker C$. So using Lemma~\ref{L:singleest}, we have
$$
\left |  \P(FX=A) - |G|^{-1} \right| =\left|  \E (\zeta^{C(-A)})\prod_{1\leq k\leq n}  \E(\zeta ^{C(v_i)X_{i}}  ) \right| \leq \exp(-\epsilon\delta n/a^2).
$$
\end{proof}

We then put these estimates for columns together using a simple inequality.

\begin{lemma}\label{L:1tom}
If we have integer $m\geq 2$ and  real numbers $x\geq 0$ and $y$ such that $|y|/x\leq 2^{1/(m-1)}-1$ and $x+y\geq 0$, then
$$
x^m-2mx^{m-1}|y| \leq (x+y)^m\leq x^m+2mx^{m-1}|y|.
$$
\end{lemma}
\begin{proof}  
We  can assume $x=1$ by homogeneity.  We divide into two cases based on the sign of $y$.
Then note the middle and right hand expressions are equal when $y=0$ and the derivative of the middle is at most the derivative of the right when $0\leq y\leq 2^{1/(m-1)}-1$.  A similar argument when $-1\leq y\leq 0$ compares the left and middle expressions.
\end{proof}

\begin{lemma}\label{L:FullFcode}
Let $R$, $G$, and $u$ be as above. Let $\epsilon >0$ and $\delta>0$.
Then there are $c,K>0$ such that the following holds.
Let $M\in\Hom(W,V)$ be  $\epsilon$-balanced random matrix.
 Let $F\in \Hom(V,G)$ be a code of distance $\delta n$. Let $A\in\Hom(W,G)$.
   For all $n$ we have
\begin{align*}
 \left|\P(FM=A) - |G|^{-n-u}\right| &\leq
 \frac{K\exp(-cn)}{|G|^{n+u}}.
\end{align*}
\end{lemma}

\begin{proof}
For $n$ large enough, we have
$$
\exp(-\epsilon\delta n/a^2)|G|\leq \log 2/(n+u-1) \leq 2^{1/(n+u-1)}-1,
$$
since
 $ 2^{1/(n+u-1)}-1= e^{\log 2/(n+u-1)} -1 \geq \log 2/(n+u-1)$.
So for $n$ sufficiently large, we can combine Lemma~\ref{L:Fcodecolumn} and Lemma~\ref{L:singleest} to obtain
$$
\left |  \P(FM=A) - |G|^{-n-u} \right| \leq 2(n+u)\exp(-\epsilon\delta n/a^2)|G|^{-n-u+1}.
$$
The lemma follows.
\end{proof}

So far, we have dealt with $F\in\Hom(V,G)$ that are codes.  Unfortunately, it is not sufficient to divide $F$ into codes and non-codes.  We need a more delicate division of $F$ based on the subgroups of $G$.  This division can be approximately understood as separating the $F$ based on what largest size subgroup  they are a code for.
For an integer $D$ with prime factorization $\prod_i p_i^{e_i}$, let $\ell(D)=\sum_i e_i$.
The following concept was introduced in \cite{Wood2014a}. Since $V_{\setminus\sigma}$ is a subgroup of $V$, for $F\in\Hom(V,G)$, the image $F(V_{\setminus\sigma})$ is a subgroup of $G$.

\begin{definition}
The \emph{depth} of an $F\in\Hom(V,G)$ is the maximal positive $D$ such that
there is a $\sigma\sub [n]$ with $|\sigma|< \ell(D)\delta n$ such that $D=[G:F(V_{\setminus\sigma})]$, or is $1$ if there is no such $D$. 
\end{definition}

\begin{remark}\label{R:depthcode}
In particular, if the depth of $F$ is $1$, then for every $\sigma\sub [n]$ with $|\sigma|< \delta n$,
we have that $F(V_{\setminus\sigma})=G$ (as otherwise $\ell([G:F(V_{\setminus\sigma})]])\geq 1$), and so we see that $F$ is a code of distance $\delta n$.
\end{remark}

We have a bound the number of $F$ that we have of depth $D$.

\begin{lemma}[Count $F$ of given depth, Lemma 5.2 of \cite{Wood2014a}]\label{L:countdepth}
There is a constant $K$ depending on $G$ such that if $D>1$, then number of $F\in \Hom(V,G)$ of depth $D$ is at most
$$
K\binom{n}{\lceil \ell(D)\delta n \rceil -1} |G|^n|D|^{-n+\ell(D)\delta n}.
$$
\end{lemma}

Now for each depth, we will get a bound on $\P(FM=0)$, with the smaller the depth, the better the bound.

\begin{lemma}[Bound probability for column given depth]\label{L:probdepthestimatecolumn}
Let $R$, $G$ be as above.   Let $\epsilon>0$ and $\delta>0$. If $F\in\Hom(V,G)$ has depth $D> 1$ and  $[G:F(V)]<D$, then for all $\epsilon$-balanced  random vectors $X$ in $V$
and all $n$,
$$\P(FX=0)\leq 
(1-\epsilon)\left( D|G|^{-1} +\exp(-\epsilon \delta n/a^2) \right)
.$$
\end{lemma}

\begin{proof}
Pick a $\sigma\sub [n]$ with $|\sigma|< \ell(D)\delta n$ such that $D=[G:F(V_{\setminus\sigma})].$
Let $F(V_{\setminus\sigma})=H.$  
However, since $[G:F(V)]<D$, we cannot have $\sigma$ empty. 
We have $FX=\sum_{i\not\in\sigma} F(v_i)X_i + \sum_{i\in \sigma} F(v_i)X_i.$ So
$$
\P(FX=0)=\P(\sum_{i\in\sigma} F(v_i)X_i \in H )\P(\sum_{i\not\in\sigma} F(v_i)X_i =- \sum_{i\in\sigma} F(v_i)X_i|\sum_{i\in\sigma} F(v_i)X_i \in H ).
$$
For the first factor, we note that since $[G:F(V)<D]$, there must be some $i\in\sigma$ with the reduction $F(v_i)\not= 0\in G/H$.  Thus conditioning on all other $X_k$ for $k\ne i$, by the $\epsilon$-balanced assumption on $X$, we have that $\P(\sum_{i\in \tau} F(v_i)X_i \in H )\leq 1-\epsilon$.

Then, we note that the restriction of $F$ to $V_{\setminus\sigma}$
 is a code of distance $\delta n$ in $\Hom(V_{\setminus\sigma}, H)$.
(If it were not, then by eliminating $\sigma$ and $<\delta n$ indices, we would eliminate $<(\ell(D)+1)\delta n$ indices
and have an image which was index that $D$ strictly divides, contradicting the depth of $F$.)  So conditioning on the $X_i$ with $i\in\sigma$, we can estimate the conditional probability above using Lemma~\ref{L:Fcodecolumn}:
$$
\P(\sum_{i\not\in\sigma} F(v_i)X_i =- \sum_{i\in\sigma} F(v_i)X_i|\sum_{i\in\sigma} F(v_i)X_i \in H )
\leq |H|^{-1} +\exp(-\epsilon \delta n/a^2).  
$$
The lemma follows.
\end{proof}

\begin{lemma}[Bound probability for matrix given depth]\label{L:probdepthestimate}
Let $R$, $G$, $u$ be as above.   Let $\epsilon >0$ and $\delta>0$. Then there is a real $K$ such that if $F\in\Hom(V,G)$ has depth $D> 1$ and  $[G:F(V)]<D$ (e.g. the latter is true if $F(V)=G$), then for all  $\epsilon$-balanced random matrices $M\in\Hom(W,V)$,
and all $n$,
$$\P(FM=0)\leq 
K \exp({-\epsilon n})  D^n|G|^{-n} 
.$$
\end{lemma}

\begin{proof}
By the independence of the columns of $M$, we can  take the $n+u$th power of the bound in Lemma~\ref{L:probdepthestimatecolumn}, and apply Lemma~\ref{L:1tom}.
We have, for $n$ large enough
\begin{align*}
&(1-\epsilon)^{n+u}\left( D|G|^{-1} +\exp(-\epsilon \delta n/a^2) \right)^{n+u} \\
\leq  &\exp({-\epsilon (n+u)}) 
\left( D^{n+u}|G|^{-n-u} +2(n+u)\exp(-\epsilon \delta n/a^2) D^{n+u-1}|G|^{-n-u+1} \right).
\end{align*}
The lemma follows.
\end{proof}

Now we can combine the estimates we have for $\P(FM=0)$ for $F$ of various depth with the bounds we have on the number of $F$ of each depth to obtain our main result on the moments of cokernels of random matrices.  
\begin{theorem}\label{T:MomMat}
Let $a$ be a  positive integer, and $u$ be a non-negative integer.
Let $\epsilon>0$ be a real number and $G$ a finite abelian group with exponent dividing $a$.
Then there are $c,K>0$ 
 such that the following holds.
Let $M$ be an $\epsilon$-balanced $n\times(n+u)$ random matrix with entries in $\Z/a\Z$.
\begin{align*}
\left|\E(\#\Sur(\cok(M),G))  -|G|^{-u} \right| \leq Ke^{-cn}.
\end{align*}
\end{theorem}
\begin{proof}
By Equation~\eqref{E:expandF}, we need to estimate
$
\sum_{F\in \Sur(V,G)} \P(FM=0).
$
We let $K$ change in each line, as long as it is a constant depending only on $a,u,\epsilon,G$.
Take $d<\min(\epsilon,\log(2)).$
Using Lemmas~\ref{L:countdepth} and \ref{L:probdepthestimate} we have
\begin{align*}
\sum_{\substack{F\in \Sur(V,G)\\ F\textrm{ not  code of distance $\delta n$}
} }
\P(FX=0)&\leq 
\sum_{\substack{D>1\\ D\mid\#G}} \sum_{\substack{F\in \Sur(V,G)\\ F\textrm{  depth $D$}
} }
\P(FX=0)\\
&\leq  \sum_{\substack{D>1\\ D\mid\#G}} K\binom{n}{\lceil \ell(D)\delta n \rceil -1} 
|G|^{n}D^{-n+\ell(D)\delta n} \exp(-\epsilon n) D^n|G|^{-n}\\
%
&\leq  \sum_{\substack{D>1\\ D\mid\#G}} K\binom{n}{\lceil \ell(D)\delta n \rceil -1} 
D^{\ell(D)\delta n} \exp(-\epsilon n) \\
&\leq K\binom{n}{\lceil \ell(|G|)\delta n \rceil -1} 
|G|^{\ell(|G|)\delta n} \exp(-\epsilon n) \\
 &\leq  K 
 e^{-dn},
\end{align*}
as long as we choose $\delta$ small enough.

Also, from Lemma~\ref{L:countdepth}, we can choose $\delta$ small enough so that we have
\begin{align*}
\sum_{\substack{F\in \Sur(V,G)\\ F\textrm{ not  code of distance $\delta n$}
} }
|G|^{-n-u}&\leq 
\sum_{\substack{D>1\\ D\mid\#G}} \sum_{\substack{F\in \Sur(V,G)\\ F\textrm{  depth $D$}
} }
|G|^{-n-u}\\
&\leq \sum_{\substack{D>1\\ D\mid\#G}}  K\binom{n}{\lceil \ell(D)\delta n \rceil -1} |G|^n|D|^{-n+\ell(D)\delta n}|G|^{-n}
\\
&\leq K \binom{n}{\lceil \ell(|G|)\delta n \rceil -1}2^{-n+\ell(|G|)\delta n}  \\
 &\leq  K 
 e^{-dn}.
\end{align*}

We also have 
 \begin{align*}
\sum_{\substack{F\in \Hom(V,G)\setminus \Sur(V,G)
} }
|G|^{-n-u}&\leq 
\sum_{H \textrm{ proper s.g of } G} \sum_{\substack{F\in \Hom(V,H)\
} }
|G|^{-n-u}\\
&\leq 
\sum_{H \textrm{ proper s.g of } G} |H|^{n+u}
|G|^{-n-u}\\
&\leq K 
 e^{-dn}.
\end{align*}

Then given a choice of $\delta$ that satisfies the two requirements above, using Lemma~\ref{L:FullFcode} we have a $c$ such that
\begin{align*}
\sum_{\substack{F\in \Sur(V,G)\\ F\textrm{  code of distance $\delta n$}
} }
\left| \P(FX=0) -|G|^{-m}\right|
&\leq 
 Ke^{-cn} . 
\end{align*}
If necessary, we take $c$ smaller so $c \leq d$.
In conclusion, 
\begin{align*}
&\left| \left(\sum_{F\in \Sur(V,G)} \P(FX=0) \right) -|G|^{-u} \right| =\left| \left(\sum_{F\in \Sur(V,G)} \P(FX=0) \right) -\left(\sum_{F\in \Hom(V,G)} |G|^{-n-u} \right)  \right|  \\
&\leq
\sum_{\substack{F\in \Sur(V,G)\\ F\textrm{  code of distance $\delta n$}
} }
\left| \P(FX=0) -|G|^{-n-u}\right|+
 \sum_{\substack{F\in \Sur(V,G)\\ F\textrm{ not  code of distance $\delta n$}
} }
 \P(FX=0)  + \sum_{\substack{F\in \Hom(V,G)\\ F\textrm{ not  code of dist. $\delta n$}
} } |G|^{-n-u} \\
\\
&\leq Ke^{-cn}.
\end{align*}
\end{proof}


\section{Moments determine the distribution}

We use the following theorem to determine the asymptotic 
distribution of $\cok(M)$ as $n\ra\infty$ from the moments in Theorem~\ref{T:MomMat}.

\begin{theorem}[c.f. Theorem 8.3 in \cite{Wood2014a}]\label{T:MomDetDetail}
Let $X_n$ and $Y_n$ be sequences of random finitely generated abelian groups.
Let $a$ be a positive integer and $A$ be the set of (isomorphism classes of) abelian groups with exponent dividing $a$.
Suppose that for every $G\in A$, we have a number $M_G\leq |\wedge^2 G|$ such that
$$
\lim_{n\ra \infty} \E(\# \Sur(X_n, G)) = M_G. 
$$
Then for every $H\in A$,
 the limit
$
\lim_{n\ra\infty} \P(X_n\tensor \Z/a\Z \isom H)
$
exists, and for all $G\in A$ we have
$$
\sum_{H\in A} \lim_{n\ra\infty} \P(X_n\tensor \Z/a\Z \isom H) \#\Sur(H,G)=M_G.
$$
If for every $G\in A$, we also have
$
\lim_{n\ra \infty} \E(\# \Sur(Y_n, G)) = M_G,
$
then, we have that for every every $H\in A$
$$
\lim_{n\ra\infty} \P(X_n\tensor \Z/a\Z \isom H) =\lim_{n\ra\infty} \P(Y_n\tensor \Z/a\Z \isom H).
$$
\end{theorem}

For the rest of this section, we fix a non-negative integer $u$.  We construct a random abelian group according to Cohen and Lenstra's distribution for each $u$ as follows.  Let $P$ be the set of primes dividing $a$.  Independently for each $p$, we have a random finite abelian $p$-group $Y_p$ given by taking each group $B$ with probability
$$
\frac{\prod_{k=1}^\infty (1-p^{-k-u})}{|B|^u|\Aut(B)|}.
$$
We then form a random group $Y$ by taking $Y=\prod_{p\in P} Y_p$. 

\begin{lemma}\label{L:CLmom} 
For every finite abelian group $G$ with exponent dividing $a$, we have
\begin{equation*}
\E(\#\Sur(Y,G))=|G|^{-u}.
\end{equation*}
\end{lemma}
In particular, taking $G$ the trivial group says that the above distribution on $B$ is a probability distribution.   

\begin{proof}
By factoring over primes $p\in P$, we can reduce to the case when $P=\{p\}$.
Let $\mathcal{A}$ be the set of finite abelian $p$-groups.
 Multiplying \cite[Proposition 4.1 (ii)]{Cohen1984} (for $k=\infty$ and $K=A$) by $|\Aut(K)|$, we obtain, for every $i$,
$$
\sum_{B\in \mathcal{A}, |B|=p^i}  \frac{|\Sur(B,G)|}{|\Aut(B)|}= \sum_{B\in \mathcal{A}, |B|=p^i/|G|} \frac{1}{|\Aut(B)|}.
$$ 
Dividing by $p^{iu}$ and summing over all $i$, we obtain
$$
\sum _{B\in \mathcal{A}} \frac{|\Sur(B,G)|}{|B|^{u}|\Aut(B)|}=|G|^{-u} \sum _{B\in \mathcal{A}} \frac{1}{|B|^u|\Aut(B)|}.
$$
By \cite[Corollary 3.7 (i)]{Cohen1984} (with $s=u$ and $k=\infty$), we have
$
\sum _{B\in \mathcal{A}} |B|^{-u}|\Aut(B)|^{-1} =\prod_{j\geq 1} (1-p^{-j-u})^{-1}, 
$
and the lemma follows.
\end{proof} 

We can now determine the distribution of our cokernels by comparing their moments to those of $Y$.


\begin{corollary}[of Theorem~\ref{T:MomMat} and Theorem \ref{T:MomDetDetail}] \label{C:first}
Let $\epsilon>0$ and let $M\in M_{n\times (n+u)}(\Z)$ (resp, $M\in M_{n\times (n+u)}(\Z_p)$) be an $\epsilon$-balanced random matrix. Let $G$ be a finite abelian group with exponent dividing $a$ (resp., $p^k$). For $Y$ defined  above,
$$
\lim_{n\ra\infty} \P(\cok(M)\tensor \Z/a\Z \isom G) = \P(\cok(Y)\tensor \Z/a\Z \isom G).
$$
\end{corollary}

In particular, we can conclude the following, which proves Theorems~\ref{T:MainIntro} and \ref{T:u}.
\begin{corollary}\label{C:Main}
Let $\epsilon>0$ and let $M\in M_{n\times (n+u)}(\Z)$ (resp, $M\in M_{n\times (n+u)}(\Z_p)$) be an $\epsilon$-balanced random matrix.
Let $B$ be a finite abelian group (resp., finite abelian $p$-group).  Let $P$ be a finite set of primes including all those dividing $|B|$ (resp., $P=\{p\})$.
Let $H_P:=\prod_{p\in P} H_p$.
  Then
\begin{align*}
\lim_{n\ra\infty} \P(\cok(M)_P\isom B) 
&= \frac{1}{|B|^u|\Aut(B)|}
\prod_{p\in P}
\prod_{k=1}^\infty (1-p^{-k-u}).
\end{align*}
\end{corollary}
\begin{proof}
Note that if $B$ is a finite abelian group with exponent that has prime factorization $\prod_{p\in P} p^{e_p}$, then if we take $a=\prod_{p\in P} p^{e_p+1}$, for any finitely generated abelian group $H$, we have
$
H\tensor \Z/a\Z \isom G$  if and only if $ H_P \isom G$.

So the corollary follows from Corollary~\ref{C:first} and the construction of $Y$
\end{proof}

Also taking $a=p$ for a prime $p$ in Corollary~\ref{C:first}, we conclude the following on the distribution of $p$-ranks.  
\begin{corollary}\label{C:prank}
Let $p$ be a prime and $\epsilon >0$.
Let $\epsilon>0$ and let $M\in M_{n\times (n+u)}(\Z/p\Z)$  be an $\epsilon$-balanced random matrix.
 For every non-negative integer $k$
$$
\lim_{n\ra\infty} \P(rank(M)=n-k) =p^{-k(k+u)}\prod_{i=1}^{k} (1-p^{-i})^{-1} \prod_{i=1}^{k+u} (1-p^{-i})^{-1}\prod_{i\geq 1} (1-p^{-i}) 
$$
\end{corollary}
\begin{proof}
We apply Theorem~\ref{T:MomMat} with $a=p$ and Theorem~\ref{T:MomDetDetail} with $a=p$ to $\cok(M)$ and $Y$.
We can read off the rank distribution of $Y$ from \cite{Cohen1984}[Theorem 6.3].  (Alternatively, instead of $Y$ we could use cokernels of $H_n\in M_{n\times (n+u)}(\Z/p\Z)$ from the uniform distribution and use the elementary count of matrices over $\Z/p\Z$ of a given rank.)
\end{proof}

\subsection*{Acknowledgements} 
This work was done with the support of an American Institute of Mathematics Five-Year Fellowship and National Science Foundation grants DMS-1301690.  The author thanks John Voight for useful comments on an earlier draft of this paper.

\def\cprime{$'$}

\end{document}